\documentclass[10pt]{amsart}
\usepackage[latin1] {inputenc}
\usepackage{fancyhdr}
\usepackage{amsmath}
\usepackage{amsfonts}
\usepackage{amssymb}
\usepackage{amsthm}
\usepackage{url}
\usepackage{tikz}
\usepackage{graphicx}

\newtheorem{Theorem}{Theorem}[subsection]
\newtheorem{Lemma}[Theorem]{Lemma}
\newtheorem{Prop}[Theorem]{Proposition} 
\newtheorem*{thma}{Theorem}
\theoremstyle{definition} \newtheorem*{step1}{Step 1}
\theoremstyle{definition} \newtheorem*{step2}{Step 2}
\theoremstyle{definition} \newtheorem*{stepi}{Step $i+1$}
\theoremstyle{definition} 
\theoremstyle{definition} 
\theoremstyle{definition} \newtheorem{definition}[Theorem]{Definition}
\theoremstyle{definition} \newtheorem{Remark}[Theorem]{Remark}
\theoremstyle{definition} \newtheorem*{Remark*}{Remark}
\theoremstyle{definition} \newtheorem{Question}[Theorem]{Question}
\theoremstyle{definition} \newtheorem*{Notation}{Notation}
\theoremstyle{definition} \newtheorem{Observation}[Theorem]{Observation}

\begin{document}
\title{Compactifying the relative Picard functor over degenerations of varieties}
\author{Atoshi Chowdhury}
\address{Instituto Nacional de Matem\'{a}tica Pura e Aplicada, Estrada Dona Castorina 110, 22460-320, Rio de Janeiro, RJ, Brazil}
\email{atoshi@impa.br}

\begin{abstract}
Over a family of varieties with singular special fiber, the relative Picard functor (i.e. the moduli space of line bundles) may fail to be compact.  We propose a stability condition for line bundles on reducible varieties that is aimed at compactifying it.  This stability condition generalizes the notion of `balanced multidegree' used by Caporaso in compactifying the relative Picard functor over families of curves.  Unlike the latter, it is defined `asymptotically'; an important theme of this paper is that although line bundles on higher-dimensional varieties are more complicated than those on curves, their behavior in terms of stability asymptotically approaches that of line bundles on curves.

Using this definition of stability, we prove that over a one-parameter family of varieties having smooth total space, any line bundle on the generic fiber can be extended to a unique semistable line bundle on the (possibly reducible) special fiber, provided the special fiber is not too complicated in a combinatorial sense.
\end{abstract}

\maketitle

\section{Introduction} \label{section-intro}

\subsection{Overview of problem}

This paper addresses the problem of compactifying the relative Picard functor---the functor parametrizing line bundles---over families of varieties in which some fibers may be singular.  This problem has been studied extensively for curves and somewhat sporadically for higher-dimensional varieties (see Section \ref{subsection-background} for some references).  Here we consider it for varieties of arbitrary dimension.

There are two obstacles that may prevent the relative Picard functor from being proper.  First, over a family of varieties whose special fiber is singular, the space of line bundles may not be closed: it may not be possible to extend a given line bundle on the generic fiber to a line bundle on the special fiber.  Second, the space may not be separated: a line bundle on the generic fiber may have more than one extension to the special fiber.

We study the second obstacle, nonseparatedness, in the following situation: let $\mathfrak{X} \rightarrow S$ be a one-parameter family of varieties whose total space is smooth, and whose special fiber $X$ may be reducible (with simple normal crossings singularities).  Over such a family, as explained in Section \ref{section-twist} below, the nonseparatedness of the relative Picard functor arises precisely from the reducibility of the special fiber $X$; in particular, if $X$ is reducible, then any line bundle on the generic fiber has infinitely many extensions to $X$.

To correct this, we propose a stability condition for line bundles on possibly reducible varieties (Definition \ref{def-stability}).  The upshot of our main results (Theorems \ref{thm-tree} and \ref{thm-dimd}) is the following (stated more precisely later as Theorem \ref{thm-upshot}):
\begin{thma}
Assume the dual graph of $X$ (appropriately defined) is a tree, and the canonical bundle of $X$ is nonzero in a suitable sense (for example, either ample or anti-ample).  Then any line bundle on the generic fiber of $\mathfrak{X}$ can be extended to a semistable line bundle on $X$.  Generically there is a unique such extension; in special cases there may be more than one, but never more than $2^{n-1}$, where $n$ is the number of irreducible components of $X$.
\end{thma}
A similar result for Type II K3 surfaces is proved in the forthcoming note \cite{k3}.

\vspace{.2in}

\subsection{Background and context} \label{subsection-background}

Over families of curves, a number of compactifications of the relative Picard scheme and of moduli spaces of higher-rank vector bundles have been constructed (e.g. \cite{odaseshadri}, \cite{caporaso}, \cite{esteves}, \cite{pandharipande}, \cite{nagarajseshadri}, \cite{schmitt}).  In the higher-dimensional setting, \cite{altmankleiman} constructs compactified Picard schemes for families of irreducible varieties; there is also work on compactifying moduli spaces of higher-rank vector bundles on single smooth surfaces rather than over degenerations (e.g. \cite{gieseker77}, \cite{maruyama}, \cite{timofeeva}, \cite{mtt}), which is thematically related to our situation in that nonseparatedness arises over vector bundles of rank at least $2$ in much the same way as it does over reducible varieties in a family.

The stability condition we define generalizes the one that appears in Caporaso's compactification (using geometric invariant theory) of the universal Picard variety over the moduli space of stable curves \cite{caporaso}.  In that compactification, the fiber over a given stable curve $X$ parametrizes semistable line bundles (in Caporaso's terminology, line bundles of \emph{balanced multidegree}) on certain semistable models (the so-called \emph{quasistable models}) of $X$.  Our stability condition is aimed at producing similar compactifications over moduli spaces of higher-dimensional varieties.

More precisely: to compactify the relative Picard functor, one must overcome the first obstacle to properness mentioned above (that a line bundle on the generic fiber of a family of varieties may have no limit over the special fiber).  There are two natural approaches to this.  One approach (used e.g. in \cite{pandharipande} and \cite{esteves}) is to construct a space whose fiber over a given singular variety $X$ parametrizes a broader class of sheaves (not necessarily locally free) on $X$.  The other (used in \cite{caporaso}) is to let the fiber over $X$ parametrize line bundles on modifications of $X$.

The results of this paper are intended to be applied in a construction using the second approach.  Specifically, suppose $\mathfrak{X} \rightarrow S$ is a one-parameter family of varieties with singular special fiber $X$ and singular total space.  Then, given a line bundle on the generic fiber of $\mathfrak{X}$ (which may have no line bundle extensions to $X$), one can always construct a desingularization $\mathfrak{X}' \rightarrow \mathfrak{X}$ and then extend the line bundle to a line bundle on the special fiber $X'$ of $\mathfrak{X}'$.  Now our results can be applied over $\mathfrak{X}'$ to count how many such extensions are semistable.

\begin{Remark}
Our ultimate hope is to use Theorem \ref{thm-upshot} to show that certain moduli spaces of semistable line bundles are proper (or at least `weakly proper', in the sense of \cite{alper}).  There is, however, considerable work to be done before this can be achieved.  The reason is indicated in the preceding paragraph: there are infinitely many distinct models $X'$ of $X$ coming from desingularizations of the original family $\mathfrak{X}$ and its base changes.  A moduli space parametrizing line bundles on the fibers of $\mathfrak{X}$ must (a priori) include line bundles on all of these models, if the first obstacle to properness (nonexistence of limits) is to be prevented.  But such a moduli space fails to be of finite type over the base.

In the curve case, it turns out that a proper moduli space can be constructed using only a finite set of modifications -- namely, the set of quasistable models -- of each $X$ (\cite{caporaso}, \cite{estevespacini}).  The generalization of this fact to higher dimension (via an appropriate generalization of the notion of quasistable curves) is the subject of ongoing joint work of the author and Eduardo Esteves.
\end{Remark}

\vspace{.2in}

\subsection{Outline of paper} \label{subsection-outline}

In Section \ref{section-twist} we examine the problem of extending line bundles over nonsingular one-parameter families of varieties and formulate the specific question that the rest of the paper will answer (Question \ref{question-twists}).

In Section \ref{section-stability} we introduce a stability condition for line bundles on possibly reducible varieties (Definition \ref{def-stability}) and prove a useful identity related to it (Proposition \ref{prop-identity}).

In Section \ref{section-formal} we establish a criterion (Theorem \ref{thm-tree}) that guarantees existence and weak uniqueness of semistable limits over varieties whose irreducible components are arranged in a combinatorially simple way.

In Section \ref{section-curve} we establish exactly when the criterion of the previous section is fulfilled by curves.  This is significantly easier than the higher-dimensional case, but illustrates the essential behavior of the latter; in particular, the stability condition in the higher-dimensional case behaves `in the limit' as it does in the curve case.

In Section \ref{section-arbitrary}, we show that the criterion of Section \ref{section-formal} is fulfilled by varieties of arbitrary dimension with ample or anti-ample canonical bundle.

\vspace{.2in}

\subsection{Acknowledgments}

I am deeply grateful to Ravi Vakil, my doctoral adviser, who introduced me to the problem considered in this paper and gave me a great deal of guidance in investigating it.  (The case $d=2$, $n=2$ of Theorem \ref{thm-upshot} appeared as the main result in my 2012 Stanford Ph.D. thesis.)

I am also grateful to Lucia Caporaso, Jesse Kass, Martin Olsson, and Brian Osserman for several helpful discussions, and to Brendan Hassett for the idea of extending this work to K3 surfaces.

\vspace{.2in}

\section{Nonseparatedness and twisting by components} \label{section-twist}

In this section we explain how nonseparatedness of the relative Picard functor arises over smooth one-parameter families of varieties.

\begin{definition} \label{def-oneparam}
A \emph{one-parameter family} of varieties of dimension $d$ is a morphism $\mathfrak{X} \rightarrow S$ with the following properties:
\begin{itemize}
\item $S$ is a smooth proper curve over a field $k$;
\item $\mathfrak{X} \rightarrow S$ is a flat proper morphism of relative dimension $d$;
\item the fibers of $\mathfrak{X} \rightarrow S$ are connected, reduced varieties (possibly reducible).
\end{itemize}

A \emph{smooth one-parameter family} is one whose total space is smooth over $k$.
\end{definition}

For the rest of this section, let $\mathfrak{X} \rightarrow S$ be a smooth one-parameter family of varieties of dimension $d$, let $s \in S$ be a closed point, and let $X$ denote the fiber of the map $\mathfrak{X} \rightarrow S$ over the point $s$.  Let $\mathfrak{X}^* = \mathfrak{X} \backslash X$, and suppose we are given a line bundle $\mathcal{L}^*$ on $\mathfrak{X}^*$.

Over the family $\mathfrak{X} \rightarrow S$, the relative Picard functor classifies equivalence classes of line bundles on $\mathfrak{X}$ modulo tensoring by pullbacks of line bundles on $S$.  (The latter operation leaves unchanged the restriction of the line bundle to each of the fibers of the family.)  Therefore, to compactify the functor, we need to understand how many different ways there are to extend $\mathcal{L}^*$ to a line bundle on $\mathfrak{X}$, up to tensoring by pullbacks of line bundles on $S$.

Since $\mathfrak{X}$ is nonsingular, there is at least one such extension; call it $\mathcal{L}$.  All other extensions can then be classified as follows.  Let $X_1,\ldots,X_n$ be the irreducible components of $X$.  For each $i$, $O_\mathfrak{X}(X_i)$ is a line bundle on $\mathfrak{X}$ that is trivial when restricted to $\mathfrak{X}^*$, so for any integers $a_1,\ldots,a_n$, the line bundle $\mathcal{L} \otimes O_\mathfrak{X}(\sum a_i X_i)$ again restricts to $\mathcal{L}^*$ on $\mathfrak{X}^*$.

Conversely, any line bundle $\mathcal{L}'$ on $\mathfrak{X}$ such that $\mathcal{L}'|_{\mathfrak{X}^*} \cong \mathcal{L}^*$ must differ from $\mathcal{L}$ by a line bundle that is trivial on $\mathfrak{X}^*$, which means the Weil divisor corresponding to the difference must be supported on $X$.  Therefore $\mathcal{L}' \cong \mathcal{L} \otimes O_\mathfrak{X}(\sum a_i X_i)$ for some integers $a_1,\ldots,a_n$.

Moreover, two extensions $\mathcal{L}$ and $\mathcal{L}'$ are equivalent modulo tensoring by pullbacks of line bundles on $S$ if and only if they differ by a multiple of $O_\mathfrak{X}(X)$ (which is the pullback of $O_S(s)$).  Since $O_\mathfrak{X}(X)|_X \cong O_X$, this occurs if and only if $\mathcal{L}|_X \cong \mathcal{L}'|_X$.

Therefore the line bundles on $\mathfrak{X}$ that restrict to $\mathcal{L}^*$ on $\mathfrak{X}^*$ are in bijection with the integer linear combinations of the irreducible components of $X$, i.e., with $\mathbf{Z}^n$.  
And the equivalence classes classified by the relative Picard functor---or, equivalently, the `limit line bundles' on $X$---are in bijection with the elements of the quotient $\mathbf{Z}^n / \langle (1,\ldots,1) \rangle$.

The crucial operation of tensoring by line bundles of the form $O_\mathfrak{X}(\sum a_i X_i)|_X$ will be called `twisting':
\begin{definition}
If $L$ and $L'$ are line bundles on $X$ such that \[ L \cong L' \otimes O_\mathfrak{X}(\sum a_i X_i)|_X \] for some $a_1,\ldots,a_n \in \mathbf{Z}$, then we say $L$ is a \emph{twist} of $L'$.
\end{definition}
\noindent (Note that the line bundle $O_\mathfrak{X}(X_i)|_X$ is independent of the family $\mathfrak{X}$; cf. Observations \ref{obs-1} and \ref{obs-2}.)

The question of how many extensions of $\mathcal{L}^*$ on $X$ (up to tensoring by pullbacks of line bundles on $S$) are semistable thus reduces to the following:
\begin{Question} \label{question-twists}
Given a line bundle $L$ on $X$, how many twists of $L$ are semistable?
\end{Question}
\noindent This is the question that the results of Sections \ref{section-formal}-\ref{section-arbitrary} answer.

\vspace{.2in}

\section{The stability condition} \label{section-stability}

In this section we propose a stability condition for line bundles on possibly reducible varieties of arbitrary dimension.  In Section \ref{subsection-definition} we give the definition in its most general form.  Then in Section \ref{subsection-smoothdeformation}, we specialize to the case of varieties that appear as fibers in smooth one-parameter families.  In Section \ref{subsection-identity} we prove the identity that underlies the usefulness of the criterion in Section \ref{section-formal}.

\vspace{.2in}

\subsection{Definition of semistability} \label{subsection-definition}

Here and throughout this paper, a \emph{variety of dimension $d$} will be a connected, reduced, but not necessarily irreducible scheme that has dimension $d$ and is proper over a field $k$, and furthermore it will be assumed to have simple normal crossings singularities (that is, each irreducible component is smooth, and irreducible components intersect transversely).  (Occasionally in remarks we will consider more general classes of varieties.)

We start by defining the main ingredient in our stability condition.

\begin{definition}  \label{def-eY}
Let $X$ be a variety of dimension $d$, and let $L$ be a line bundle on $X$.  For each union $Y \subset X$ of irreducible components of $X$, let $D = Y \cap \overline{X \backslash Y}$, and let 
\begin{equation} \label{eqn-eY}
e_Y(L) = d! \left( \chi(X,L) \cdot \frac{1}{d+1} \sum_{j=1}^{d+1} \binom{d+1}{j} [L^{d+1-j} D^{j-1} Y] - [L^d X] \cdot \chi(Y,L) \right) \text{.}
\end{equation}
\end{definition}

\begin{Notation}
The expressions $[L^{d+1-j} D^{j-1} Y]$ and $[L^d X]$ in (\ref{eqn-eY}) need some explanation.  In general, throughout this paper, quantities in square brackets represent intersection numbers.  Both within square brackets and outside them, we conflate the notation for Weil divisors, Cartier divisors, line bundles, and their first Chern classes wherever this can be done without ambiguity.  For all of these objects, multiplication denotes the intersection product, and addition denotes addition of divisors or tensor product of line bundles.

The intersection number $[L^{d+1-j} D^{j-1} Y]$ in (\ref{eqn-eY}) should be interpreted as follows: the $Y$ in the brackets serves to indicate that the intersection product is computed on $Y$.  The $L$ thus represents the restriction of the line bundle $L$ to $Y$; the $D$ represents the Weil divisor $D$ on $Y$ (which is also Cartier, thanks to the simple normal crossings assumption), or equivalently the line bundle $O_Y(D)$.

The term $[L^d X]$ in (\ref{eqn-eY}) represents the degree of $L$ on $X$.  The $X$ inside the brackets serves to indicate that the intersection $L^d$ is computed on $X$.

In Section \ref{subsection-smoothdeformation} we'll see that when $X$ is the special fiber in a smooth one-parameter family (which is the situation in which the results of Sections \ref{section-formal}-\ref{section-arbitrary} apply), these intersection products can be written and interpreted in a more unified way.
\end{Notation}

\begin{Remark}
Dividing the right-hand side of (\ref{eqn-eY}) through by $d! \chi(X,L) [L^dX]$ produces a more symmetric-looking formula and does not substantially affect the results in the rest of the paper, but we retain the version given because it turns out to be slightly more convenient for calculations.
\end{Remark}

\begin{Remark}
Definition \ref{def-eY} makes sense for somewhat more general varieties $X$: all that is really required is that the irreducible components of $X$ intersect transversely and that the Weil divisor $D = Y \cap \overline{X \backslash Y}$ be $\mathbf{Q}$-Cartier on each $Y$.  But we won't use this level of generality.
\end{Remark}

\begin{Remark} \label{remark-git}
The key property of the formula defining $e_Y(L)$ is that stated in Proposition \ref{prop-identity}; however, the formula was originally obtained without reference to such a property.  Instead, its derivation was motivated by geometric invariant theory, as follows.

In \cite{caporaso}, Caporaso constructs a compactified Picard scheme for families of curves by taking a GIT quotient of a Hilbert scheme; she identifies the GIT-semistable points of the Hilbert scheme as precisely those corresponding to `balanced' line bundles (see Definition \ref{def-balanced}).  This identification requires many technical results that cannot readily be generalized to higher-dimensional varieties.  But part of it -- the fact that if a point of the Hilbert scheme is GIT-semistable, then the associated line bundle is balanced (Proposition 3.1 in \cite{caporaso}, based on work of Gieseker \cite{gieseker82}) -- rests on arguments that, while given specifically for nodal curves, are not really curve-specific.

We derived the formula for $e_Y(L)$ heuristically by adapting these arguments (ignoring many technical issues) to varieties of arbitrary dimension with simple normal crossings singularities.
\end{Remark}

We define semistability of line bundles via the asymptotic behavior of the functions $e_Y$:

\begin{definition} \label{def-stability}
Let $X$ be a variety of dimension $d$, and let $L$ and $H$ be line bundles on $X$.  We say $L$ is \emph{$H^-$-semistable} (resp. \emph{$H^+$-semistable}) if for every sufficiently large positive integer $m$, $e_Y(L+mH) \leq 0$ (resp. $e_Y(L+mH) \geq 0$) for every union of irreducible components $Y \subset X$.
\end{definition}

\begin{Remark}
If $X$ is irreducible, every line bundle on $X$ is trivially both $H^-$- and $H^+$-semistable for any $H$.
\end{Remark}

\begin{Remark}
In Section \ref{section-arbitrary} we will work over varieties $X$ for which $\pm K_X$ is positive in a suitable sense, and count $K_X^-$- or $K_X^+$-semistable line bundles.  In other situations, such as when $X$ is a K3 surface (considered in \cite{k3}), it is useful to consider $H^-$- and $H^+$-semistable line bundles for different choices of $H$.
\end{Remark}

\begin{Remark} \label{remark-caporaso}
Definition \ref{def-stability} generalizes the stability condition given in \cite{caporaso} for line bundles over curves, which is as follows:

\begin{definition} \label{def-balanced}
Let $X$ be a nodal curve of (arithmetic) genus at least $2$.  A line bundle $L$ on $X$ is \emph{balanced} if for every union $Y \subset X$ of irreducible components of $X$,
\begin{equation} \label{eqn-balanced}
d_Y \geq \frac{d_X}{g_X-1}\left(g_Y-1+\frac{k_Y}{2}\right)-\frac{k_Y}{2} \text{,}
\end{equation}
where $d_X$ is the degree of $L$ on $X$, $d_Y$ is the degree of $L$ restricted to $Y$, $g_X$ is the genus of $X$, $g_Y$ is the genus of $Y$, and $k_Y$ is the number of points in which $Y$ meets $\overline{X \backslash Y}$.
\end{definition}

It is straightforward to verify that a line bundle $L$ on a curve $X$ of genus at least $2$ is balanced if and only if it is $K_X^-$-semistable; the inequality (\ref{eqn-balanced}) is equivalent to the inequality $e_Y(L) \leq 0$ when $d=1$.  (The number $k_Y$ is precisely the degree of the divisor $D$ in Definition \ref{def-eY}.)  Furthermore, the asymptotic element of Definition \ref{def-stability} is actually superfluous when applied to line bundles on curves, whereas it is essential when we consider varieties of dimension at least $2$; see Remarks \ref{remark-curveasymptotic} and \ref{remark-surfaceasymptotic}.
\end{Remark}

\vspace{.2in}

\subsection{Varieties with smooth deformation} \label{subsection-smoothdeformation}

When the variety $X$ is the special fiber in a smooth one-parameter family, the intersection numbers appearing in Definition \ref{def-eY} have a unified interpretation and satisfy some important relations, which we explain here.

\begin{definition} \label{def-smoothdef}
Let $X$ be a variety of dimension $d$.  We say $X$ \emph{has a smooth deformation} if there exists a smooth one-parameter family $\mathfrak{X} \rightarrow S$ of varieties of dimension $d$ in which one of the closed fibers is isomorphic to $X$.
\end{definition}

\begin{Remark}
The total space $\mathfrak{X}$ in Definition \ref{def-smoothdef} could be assumed to be merely factorial (or even $\mathbf{Q}$-factorial, if we keep track of denominators when calculating intersection numbers) without affecting the discussion or results in the rest of the paper.
\end{Remark}

If $X$ is a variety that has a smooth deformation $\mathfrak{X} \rightarrow S$, there are two important consequences:

\begin{Observation} \label{obs-1}
Let $Y \subsetneq X$ be a union of irreducible components of $X$, and let $Z = \overline{X \backslash Y}$.  Let $D = Y \cap Z$.  Then $Z$ is a Cartier divisor on $\mathfrak{X}$, and $O_Y(D) = O_\mathfrak{X}(Z)|_Y$.
\end{Observation} 

\begin{Observation} \label{obs-2}
Since $O_\mathfrak{X}(Y+Z)|_X = O_\mathfrak{X}(X)|_X \cong O_X$, we have $O_\mathfrak{X}(Z)|_X \cong (O_\mathfrak{X}(Y)|_X)^\vee$.  Equivalently, the normal bundles of $D$ in $Y$ and in $Z$ are dual to each other.  Another interpretation of this fact: in the Chow ring of $\mathfrak{X}$, we have the relations $D = YZ = -Y^2 = -Z^2$.
\end{Observation}

These observations allow us to reinterpret the intersection number $[L^{d+1-j}D^{j-1}Y]$ of (\ref{eqn-eY}) in a few different ways.  The cycle $[D^{j-1}]$ on $Y$ can be written as a cycle on $\mathfrak{X}$, namely $[Z^{j-1}Y]$.  By Observation \ref{obs-2}, in the Chow ring of $\mathfrak{X}$ this is the same as $(-1)^j[Y^{j-1}Z]$, which equals the cycle $(-1)^j [D^{j-1}]$ computed on $Z$.  Since this cycle in the Chow ring of $\mathfrak{X}$ is supported on $X$ (in fact on $D$), it can be intersected with $L$.

Therefore, for each $j \geq 2$, the intersection number $[L^{d+1-j}D^{j-1}Y]$ can be interpreted as an intersection number on $Y$ (namely $[L^{d+1-j}D^{j-1}]$) or on $Z$ (namely $(-1)^j [L^{d+1-j}D^{j-1}]$).  If $L$ happens to be the restriction to $X$ of a line bundle $\mathcal{L}$ on $\mathfrak{X}$, then it also equals $(-1)^{j-1}[\mathcal{L}^{d+1-j}Y^j]$.

All of these interpretations can be unified by writing the intersection number $[L^{d+1-j}D^{j-1}Y]$ as $(-1)^{j-1}[L^{d+1-j}Y^j]$ (for $j \geq 2$).  In this notation, the factor $Y^j$ can be interpreted as a cycle on $\mathfrak{X}$; in particular, the relations $YZ = -Y^2 = -Z^2$ hold.  
The terms $[L^dY]$ and $[L^dX]$ in (\ref{eqn-eY}) can similarly be interpreted in the Chow ring of $\mathfrak{X}$, which motivates our notation.

(Note moreover that our observations imply that the line bundle $O_\mathfrak{X}(Y)|_X$ is independent of the choice of family $\mathfrak{X}$.)

Thus we can write $e_Y(L)$ in the following form, which will be more convenient to use than (\ref{eqn-eY}):

\begin{Prop} \label{prop-cleaneY}
Suppose $X$ is a variety of dimension $d$ that has a smooth deformation.  Let $Y$ be a union of irreducible components of $X$, and let $L$ be any line bundle on $X$.  Then
\[
e_Y(L) = d! \left( \chi(X,L) \cdot \frac{1}{d+1} \sum_{j=1}^{d+1} \binom{d+1}{j} (-1)^{j-1} [L^{d+1-j} Y^{j}] - [L^d X] \cdot \chi(Y,L) \right) \text{,}
\]
where the intersection number $[L^{d+1-j} Y^{j}]$ is interpreted as above.
Furthermore, if $L$ is the restriction of a line bundle $\mathcal{L}$ on $\mathfrak{X}$, then 
\[
e_Y(L) = d! \left( \chi(X,L) \cdot \frac{1}{d+1} \left( [\mathcal{L}^{d+1}]-[(\mathcal{L}-Y)^{d+1}] \right) - [L^d X] \cdot \chi(Y,L) \right) \text{.}
\]
\end{Prop}

\vspace{.2in}

\subsection{A fundamental identity} \label{subsection-identity}

The following identity is at the heart of the results of Sections \ref{section-curve} and \ref{section-arbitrary}:

\begin{Prop} \label{prop-identity}
Let $X$ be a variety of dimension $d$, let $Y$ be a union of irreducible components of $X$, let $Z = \overline{X \backslash Y}$ denote the union of the irreducible components not in $Y$, and let $L$ be any line bundle on $X$.  If $X$ has a smooth deformation, then $e_Z(L) = -e_Y(L+Y)$.
\end{Prop}

\begin{Remark} \label{remark-dual}
In fact $X$ does not need to have a smooth deformation for this identity to hold; it suffices that the normal bundles of $D$ in $Y$ and $Z$ be dual to each other, as in Observation \ref{obs-2} in Section \ref{subsection-smoothdeformation}.  This happens if, for example, $X$ has a one-parameter deformation whose total space is merely $\mathbf{Q}$-factorial.
\end{Remark}

\begin{proof}
By Proposition \ref{prop-cleaneY} above,
\begin{equation} \label{eqn-3}
e_Z(L) = d! \left( \chi(X,L) \cdot \frac{1}{d+1} \sum_{j=1}^{d+1} \binom{d+1}{j} (-1)^{j-1} [L^{d+1-j} Z^{j}] - [L^d X] \cdot \chi(Z,L) \right) \text{.}
\end{equation}

First we deal with the series in this expression.  By the discussion in Section \ref{subsection-smoothdeformation}, 
\[
[L^{d+1-j}Z^j] = (-1)^j[L^{d+1-j}Y^j]
\]
for each $j \geq 2$, and for the $j=1$ term we have $[L^d Z] = [L^d X]-[L^d Y]$, so
\begin{align*}
\frac{1}{d+1} \sum_{j=1}^{d+1} \binom{d+1}{j} (-1)^{j-1} [L^{d+1-j} Z^{j}]  &= [L^dX]-\frac{1}{d+1} \sum_{j=1}^{d+1} \binom{d+1}{j} [L^{d+1-j} Y^j]\text{.}
\end{align*}
Making the substitution $L^{d+1-j} = ((L+Y)-Y)^{d+1-j}$, expanding the binomial power, and rearranging the resulting series, we obtain
\begin{align} \label{eqn-1}
\frac{1}{d+1} \sum_{j=1}^{d+1} \binom{d+1}{j} (-1)^{j-1} [L^{d+1-j} Z^{j}]
&= [L^dX] \notag
\\
&-\frac{1}{d+1} \sum_{k=1}^{d+1} \binom{d+1}{k} (-1)^{k-1} [(L+Y)^{d+1-k} Y^{k}] \text{.}
\end{align}

Next we deal with the Euler characteristic term in (\ref{eqn-3}).  Let $D = Y \cap Z$; then 
\begin{align} \label{eqn-2}
\chi(Z,L)
 &= \chi(X,L)+\chi(D,L)-\chi(Y,L) \notag
\\ &= \chi(X,L) - \chi(Y,L+Y) \text{.}
\end{align}
The first line follows from the exact sequence
\begin{align*}
&0 \rightarrow L \rightarrow L|_Y \oplus L|_Z \rightarrow L|_D \rightarrow 0 \text{}
\end{align*}
of sheaves on $X$.  The second line follows from the exact sequence 
\begin{align*}
&0 \rightarrow O_Y (-D) \rightarrow O_Y \rightarrow O_{D} \rightarrow 0 \text{}
\end{align*}
of sheaves on $Y$.

Substituting (\ref{eqn-1}) and (\ref{eqn-2}) into (\ref{eqn-3}) produces the identity $e_Z(L) = -e_Y(L+Y)$.
\end{proof}

\begin{Remark}
If $L$ is the restriction to $X$ of a line bundle $\mathcal{L}$ on $\mathfrak{X}$, then the result of (\ref{eqn-1}) can be obtained more cleanly using intersections on $\mathfrak{X}$: 
\begin{align*}
\frac{1}{d+1} \sum_{j=1}^{d+1} \binom{d+1}{j} (-1)^{j-1} [L^{d+1-j} Z^{j}] &= [\mathcal{L}^{d+1}]-[(\mathcal{L}-Z)^{d+1}]
\\
&= [(\mathcal{L})^{d+1}]-[(\mathcal{L}+Y-X)^{d+1}]
\\
&= [(\mathcal{L})^{d+1}]-([(\mathcal{L}+Y)^{d+1}]-[L^d X]) \text{.}
\end{align*}
In the final step we use the fact that for $j \geq 2$, $[L^{d+1-j} X^j] = 0$ (since $X^2$ represents the line bundle $O_\mathfrak{X}(X)|_X$, which is trivial).
\end{Remark}

\vspace{.2in}

\section{A criterion for existence and uniqueness of semistable twists} \label{section-formal}

In this section we state and prove an essentially combinatorial result, Theorem \ref{thm-tree}, which will allow us in Sections \ref{section-curve} and \ref{section-arbitrary} to deduce from Proposition \ref{prop-identity} that line bundles on certain varieties have (generically) unique semistable twists.

\vspace{.2in}

\subsection{Definitions and observations}

Let $X$ be a variety of dimension $d$, $Y$ a union of irreducible components of $X$, and $L$ a line bundle on $X$.  Observe that the function $e_Y(L+bY)$ is a polynomial (of degree at most $d$) in $b$, so it makes sense to consider the real solutions, rather than only the integer solutions, of the inequalities governing the semistability of $L+bY$.

\begin{definition} \label{def-**}
Let $X$ be a variety that has a smooth deformation, let $H$ be a line bundle on $X$, let $Y \subset X$ be a union of irreducible components of $X$, and let $Z = \overline{X \backslash Y}$.  We say the pair $(X,Y)$ is \emph{$H^-$-twistable} if for any line bundle $L$, the set of real numbers $b \in \mathbf{R}$ such that $e_Y(L+mH+bY) \leq 0$ and $e_Z(L+mH+bY) \leq 0$  for every sufficiently large positive integer $m$ is a unit interval containing at least one of its endpoints.

We say the pair $(X,Y)$ is \emph{$H^+$-twistable} if for any line bundle $L$, the set of real numbers $b \in \mathbf{R}$ such that $e_Y(L+mH+bY) \geq 0$ and $e_Z(L+mH+bY) \geq 0$  for every sufficiently large positive integer $m$ is a unit interval containing at least one of its endpoints.
\end{definition}

The fact that the real solutions of this system of inequalities constitute a unit interval containing at least one of its endpoints means that either there is a unique integer solution, or there are exactly two (which occurs if and only if the interval is closed and its endpoints are integers).  In particular, for varieties with exactly two irreducible components, it immediately implies the desired result:

\begin{Observation} \label{obs-twocomponent}
Let $X$ be a variety that has a smooth deformation, and suppose $X$ has exactly two irreducible components, $Y$ and $Z$.  If the pair $(X,Y)$ is $H^-$-twistable (resp. $H^+$-twistable), then any line bundle $L$ on $X$ has either a unique $H^-$-semistable (resp. $H^+$-semistable) twist, or exactly two.
\end{Observation}

Observation \ref{obs-twocomponent} can be generalized to any variety in which the irreducible components can be dealt with two at a time:

\begin{definition}
Let $X$ be a variety.  We define the \emph{dual graph} of $X$ to be the graph $\Gamma_X$ whose vertices are the irreducible components of $X$, and in which there is an edge between two vertices if and only if the corresponding irreducible components of $X$ have nonempty intersection.
\end{definition}

\begin{Theorem} \label{thm-tree}
Let $X$ be a variety that has a smooth deformation, and assume that its dual graph $\Gamma_X$ is a tree.  Let $H$ be a line bundle on $X$.  Suppose that for every union of irreducible components $Y \subset X$ such that both $Y$ and $\overline{X \backslash Y}$ are connected, the pair $(X,Y)$ is $H^-$-twistable (resp. $H^+$-twistable).  Then any line bundle $L$ on $X$ has at least one $H^-$-semistable (resp. $H^+$-semistable) twist, and at most $2^{n-1}$, where $n$ is the number of irreducible components of $X$.
\end{Theorem}

\begin{Remark}
When $X$ is a nodal curve, the dual graph of $X$ as defined here is the result of collapsing all multiple edges in the dual graph of $X$ as it is ordinarily defined.  In particular, the curves to which Theorem \ref{thm-tree} applies are those of \emph{pseudo-compact type}, in the terminology of \cite{osserman}.
\end{Remark}

\begin{Remark} \label{remark-strict}
The dichotomy imposed by the unit interval of Definition \ref{def-**}---unique semistable twists in the general case versus exactly two in special cases---mirrors the phenomenon of stability versus strict semistability in geometric invariant theory.  Let us explore this briefly.

By analogy with GIT, we may say a semistable line bundle $L$ on a variety $X$ is \emph{stable} if no nontrivial twist of it is semistable (that is, if, when $L$ is the limit of a family of line bundles, it is the \emph{only} limit), and \emph{strictly semistable} otherwise.  Then, under the hypotheses of Theorem \ref{thm-tree}, the semistable twists of a line bundle $L$ will be strictly semistable (that is, there will be more than one of them) if and only if for some $Y$, the unit interval of Definition \ref{def-**} contains exactly two integers.

See Remarks \ref{remark-curvestrict} and \ref{remark-dimd-strictss} for precise conditions on when this happens.
\end{Remark}

\vspace{.2in}

\subsection{Proof of Theorem \ref{thm-tree}}

We now prove Theorem \ref{thm-tree} for varieties $X$ such that each pair $(X,Y)$ with both $Y$ and $\overline{X \backslash Y}$ connected is $H^-$-twistable.  (The proof of the $H^+$-twistable version is identical.)  The proof requires the following two lemmas:

\begin{Lemma} \label{lemma-adjacent}
Let $X$ be a variety that has a smooth deformation.  Let $X$ have irreducible components $X_1,\ldots,X_n$, let $Y$ be some union of the irreducible components of $X$, and let $Z = \overline{X \backslash Y}$.  Let $J = \{i: X_i \cap Y \neq \emptyset \text{ and } X_i \cap Z \neq \emptyset\}$.  Then for any line bundle $L$ on $X$ and any $a_1,\ldots,a_n \in \mathbf{Z}$, we have 
\[
e_Y(L+\sum_{i=1}^n a_i X_i) = e_Y(L+\sum_{i \in J} a_i X_i) \text{.}
\]
\end{Lemma}

\begin{proof}[Proof of Lemma \ref{lemma-adjacent}.]
It suffices to show that for any line bundle $L$ on $X$, if $i \notin J$, then $e_Y(L+X_i) = e_Y(L)$.

If $X_i \cap Y = \emptyset$, the desired conclusion is obvious from the fact that \[ [(L+X_i)^jY^{d+1-j}] = [L^jY^{d+1-j}] \] for any $j \leq d$.

If $X_i \cap Z = \emptyset$, then by Proposition \ref{prop-identity}, $e_Y(L+X_i) = -e_Z(L+X_i+Z)$.  By the previous argument applied to $Z$, the right-hand side of this equals $-e_Z(L+Z)$, which by a second application of Proposition \ref{prop-identity} equals $e_Y(L)$.
\end{proof}

\begin{Lemma} \label{lemma-connected}
Let $X$ be a variety that has a smooth deformation.  Suppose $e_Y(L) \leq 0$ for every union of irreducible components $Y$ such that both $Y$ and $Z = \overline{X \backslash Y}$ are connected.  Then $e_Y(L) \leq 0$ for every union of irreducible components $Y$.
\end{Lemma}

\begin{proof}[Proof of Lemma \ref{lemma-connected}.]
First observe that if $Y$ is an arbitrary union of irreducible components of $X$, with connected components $Y_1,\ldots,Y_r$, then $e_Y(L) = e_{Y_1}(L)+\ldots + e_{Y_r}(L)$.  Therefore it suffices to prove that $e_Y(L) \leq 0$ for every \emph{connected} union of irreducible components $Y$.

So let $Y$ be a connected union of irreducible components of $X$, and let $Z = \overline{X \backslash Y}$ have connected components $Z_1,\ldots,Z_r$.  Note that for each $i$, $W_i := \overline{X \backslash Z_i}$ is connected (because $W_i$ contains $Y$, which intersects $Z_1,\ldots,Z_r$ since $X$ is connected).  Therefore the hypothesis of the lemma tells us that $e_{W_i}(L) \leq 0$ for every $i$.  Using Proposition \ref{prop-identity} and Lemma \ref{lemma-adjacent}, we have
\begin{align*}
e_Y(L) &= -e_Z(L+Z)
\\
&= -e_{Z_1}(L+Z) - \ldots -e_{Z_r}(L+Z)
\\
&= -e_{Z_1}(L+Z_1) - \ldots -e_{Z_r}(L+Z_r)
\\
&= e_{W_1}(L) + \ldots + e_{W_r}(L)
\\
&\leq 0 \text{.}
\end{align*}
\end{proof}

Now let $X$ and $H$ be as in the hypotheses of Theorem \ref{thm-tree}, and let $L$ be a line bundle on $X$.  We prove the theorem by first giving a recursive algorithm for finding a solution to the system of inequalities defining $H^-$-semistability for twists of $L$, and then verifying that any solution arises from this algorithm.

In the proof, we will conflate the notation for an irreducible component of $X$ and the corresponding vertex of $\Gamma_X$, and more generally for a union of irreducible components of $X$ and the corresponding subgraph of $\Gamma_X$.  Note also that we do not fix a numbering of the irreducible components of $X$ to start with; instead, the algorithm will assign a numbering to them.

Here is the algorithm.

\begin{step1}
Let $X_1$ be any vertex of $\Gamma_X$.  Let $a_1 = 0$.
\end{step1}

\begin{step2}
Let $X_2$ be any vertex of $\Gamma_X$ adjacent to $X_1$.  If we remove the edge $(X_1,X_2)$ from $\Gamma_X$, the resulting graph has two connected components; let $Y_2$ be the connected component containing $X_1$, $Z_2$ the connected component of $X_2$ (so $Z_2 = \overline{X \backslash Y_2}$).  Since $(X,Y_2)$ is $H^-$-twistable, we can find an integer $a_2$ such that
\begin{align*}
e_{Y_2}(L+mH+a_2X_2) &\leq 0
\\
e_{Z_2}(L+mH+a_2X_2) &\leq 0
\end{align*}
for every sufficiently large $m$.
\end{step2}

\begin{stepi}
Assume that in Steps $1$ through $i$ we have accomplished the following:
\begin{itemize}
\item We have chosen distinct irreducible components $X_1,\ldots, X_i$ of $X$ such that for each $j \in \{2, \ldots, i\}$, $X_j$ is adjacent to $X_{k_j}$ for some $k_j<j$.  (Note that $k_j$ is necessarily unique since $\Gamma_X$ is a tree.)
\item For each $j \in \{2, \ldots, i\}$, if we remove the edge $(X_j, X_{k_j})$ from $\Gamma_X$, the resulting graph has two connected components; we let $Y_j$ denote the connected component of $X_{k_j}$, $Z_j$ the connected component of $X_j$.
\item We have found integers $a_2,\ldots,a_i$ such that for each $j \in \{2, \ldots, i\}$,
\begin{align*}
e_{Y_j}(L+mH+a_2X_2+\ldots+a_iX_i) &\leq 0
\\
\text{and } e_{Z_j}(L+mH+a_2X_2+\ldots+a_iX_i) &\leq 0
\end{align*}
for every sufficiently large $m$.
\end{itemize}

Now choose $X_{i+1}$ to be any vertex adjacent to one of $X_1,\ldots,X_i$ (but not equal to any of them).  Let $k$ be the unique element of $\{1,\ldots,i\}$ such that $X_{i+1}$ and $X_k$ are adjacent.
Let $Y_{i+1}$ and $Z_{i+1}$ denote the connected components of $X_k$ and $X_{i+1}$ respectively in the graph obtained by deleting the edge $(X_{i+1},X_k)$.

Since $(X,Y_{i+1})$ is $H^-$-twistable, we can find an integer $a_{i+1}$ such that 
\begin{align*}
e_{Y_{i+1}}(L+a_kX_k+mH+a_{i+1}X_{i+1}) &\leq 0
\\
\text{and } e_{Z_{i+1}}(L+a_kX_k+mH+a_{i+1}X_{i+1}) &\leq 0
\end{align*}
for every sufficiently large $m$.  Since $X_{i+1}$ is not adjacent to $X_j$ for any $j \leq i$ except $j=k$, we have
\begin{align*}
e_{Y_{i+1}}(L+mH+a_2X_2+\ldots + a_{i+1}X_{i+1}) &\leq 0
\\
\text{and } e_{Z_{i+1}}(L+mH+a_2X_2+\ldots + a_{i+1}X_{i+1}) &\leq 0
\end{align*}
for every sufficiently large $m$.  
Also, for each $j \leq i$, we know that the only vertices adjacent to both $Y_j$ and $Z_j$ are $X_j$ and $X_{k_j}$;
in particular, $X_{i+1}$ is not.  So for each $j \leq i$,
\begin{align*}
e_{Y_j}(L+mH+a_2X_2+\ldots + a_{i+1}X_{i+1}) &= e_{Y_j}(L+mH+a_2X_2+\ldots + a_{i}X_{i}) \leq 0
\\
\text{and } e_{Z_j}(L+mH+a_2X_2+\ldots + a_{i+1}X_{i+1}) &= e_{Z_j}(L+mH+a_2X_2+\ldots + a_{i}X_{i}) \leq 0
\end{align*}
for every sufficiently large $m$.

\end{stepi}

We claim that after Step $n$, the numbering $X_1,\ldots,X_n$ of the vertices of $\Gamma_X$ and the integers $a_2,\ldots,a_n$ produced by the algorithm have the following property: for every union of irreducible components $Y$ such that both $Y$ and $Z=\overline{X \backslash Y}$ are connected, 
\[
e_Y(L+mH+ \sum_{i=2}^n a_i X_i) \leq 0
\] 
for every sufficiently large $m$.  Indeed, for each such $Y$, there is a unique edge $(X_k,X_{k'})$ between $Y$ and $Z$.  We may assume $X_k \subset Z$ and $k > k'$.  Then in Step $k$ of the algorithm, $k'$ must have been the unique integer less than $k$ such that there was an edge from $X_k$ to $X_{k'}$.  So $Y = Y_k$ and $Z = Z_k$.

Note also that in Step $2$ there are at most two choices for the integer $a_2$; in Step $3$, given $a_2$, there are at most two choices for $a_3$; and so on.  Therefore the algorithm produces at most $2^{n-1}$ semistable twists of $L$.

To complete the proof of Theorem \ref{thm-tree}, we verify that every semistable twist of $L$ equals one produced by the algorithm.  Suppose $b_1,\ldots,b_n$ are integers such that $L+\sum_{i=1}^n b_i X_i$ is $H^-$-semistable.  Since 
\[
\sum_{i=1}^n b_i X_i \cong \sum_{i=2}^n (b_i-b_1)X_i \text{,}
\]
this implies that, for every union of irreducible components $Y$ such that both $Y$ and $Z=\overline{X \backslash Y}$ are connected, 
\[
e_Y(L+mH+ \sum_{i=2}^n (b_i-b_1) X_i) \leq 0
\]
for every sufficiently large $m$.  Therefore, for each $j \in \{1,\ldots,n\}$,
\begin{align*}
e_{Y_j}(L+mH+\sum_{i=2}^n (b_i-b_1) X_i) &\leq 0
\\
\text{and } e_{Z_j}(L+mH+\sum_{i=2}^n (b_i-b_1) X_i) &\leq 0
\end{align*}
for every sufficiently large $m$.
Taking $j=2$ and removing the terms of the sum corresponding to irreducible components that intersect neither $Y_2$ nor $Z_2$, we find that $b_2-b_1$ must equal one of the (at most two) possible values of $a_2$ found in Step 2 of the algorithm.  Similarly, taking $j=3$ and removing extraneous terms from the sum, we find that $b_3-b_1$ must equal one of the at most two possible values that $a_3$ can have (given that $a_2 = b_2-b_1$).  Continuing in this way, we conclude that $L+\sum_{i=2}^n (b_i-b_1) X_i$ (which is isomorphic to $L+\sum_{i=1}^n b_i X_i$) is one of the twists produced by the algorithm.

\vspace{.2in}

\section{Curves} \label{section-curve}
\setcounter{Theorem}{0}

In this section we prove the following theorem for curves, as a warm-up to the higher-dimensional case.  This is essentially a rephrasing of results of \cite{caporaso}.

\begin{Theorem} \label{thm-curve}
Let $X$ be a curve of arithmetic genus $g_X$, and let $Y$ be any union of irreducible components of $X$.  If $g_X \geq 2$, then $(X, Y)$ is $K_X^-$-twistable.  If $g_X = 0$, then $(X, Y)$ is $K_X^+$-twistable.  If $g_X=1$, then $(X, Y)$ is neither $H^-$-twistable nor $H^+$-twistable, for any $H$.
\end{Theorem}

\begin{Remark}
Any nodal curve has a smooth deformation (alternatively, the criterion of Remark \ref{remark-dual} holds trivially for curves), so Proposition \ref{prop-identity} applies automatically.
\end{Remark}

We start with some general observations.  Let $X$ and $Y$ be as in Theorem \ref{thm-curve}, let $Z = \overline{X \backslash Y}$, and let $H$ and $L$ be arbitrary line bundles on $X$.  To check whether $(X,Y)$ is $H^-$- or $H^+$-twistable, we need to understand when the functions $e_Y(L+mH+bY)$ and $e_Z(L+mH+bY)$ have the same sign.  So we calculate them explicitly: letting $g_Y$ denote the genus of $Y$, we have
\begin{align*} 
e_Y(L+mH+bY) &= Ab+Bm+C
\end{align*}
where
\begin{align*}
A &= [Y^2](-g_X+1) \text{,}
\\
B &= [HY](-g_X+1) + [HX](g_Y-1-\frac12 [Y^2]) \text{,}
\\
C &= [LY](-g_X+1) + [LX](g_Y-1)-\frac12[Y^2]([LX]-g_X+1) \text{.}
\end{align*}
This is a linear function of $b$, so for each fixed $m$ we have the following three possibilities:

\begin{itemize}
\item Case (i): $A>0$.  In this case, let $t_m = -\frac{Bm+C}{A}$.  By Proposition \ref{prop-identity}, $e_Z(L+mH+bY) = -e_Y(L+mH+(b+1)Y)$, so if $b \in [t_m-1,t_m]$, then $e_Y(L+mH+bY) \leq 0$ and $e_Z(L+mH+bY) \leq 0$.  For all other values of $b$, $e_Y(L+mH+bY)$ and $e_Z(L+mH+bY)$ have opposite signs; in particular, they are never simultaneously positive.

\item Case (ii): $A=0$.  In this case, both $e_Y$ and $e_Z$ are constant with respect to $b$.

\item Case (iii): $A<0$.  In this case, again let $t_m = -\frac{Bm+C}{A}$.  Then Proposition \ref{prop-identity} implies that if $b \in [t_m-1,t_m]$, then $e_Y(L+mH+bY) \geq 0$ and $e_Z(L+mH+bY) \geq 0$.  For all other values of $b$, $e_Y(L+mH+bY)$ and $e_Z(L+mH+bY)$ have opposite signs; in particular, they are never simultaneously negative.
\end{itemize}

\begin{Observation} \label{obs-genus}
Since $-[Y^2] = [YZ]$ is the number of points in which $Y$ intersects $Z$, we have $[Y^2] < 0$.  Therefore Case (i) occurs when $g_X \geq 2$, Case (ii) when $g_X=1$, and Case (iii) when $g_X=0$.  Note that which case occurs does not depend on $L$, $H$, $m$, or $Y$, only on $X$.  (By contrast, the analogous case breakdown for higher-dimensional varieties depend on all four of these inputs as well as on $X$.)
\end{Observation}

\begin{Observation} \label{obs-B=0}
A pair $(X,Y)$ is $H^-$-twistable (resp. $H^+$-twistable) if and only if Case (i) (resp. Case (iii)) occurs and $\lim_{m \rightarrow \infty} t_m$ is finite.  The latter occurs if and only if $B=0$ (in which case $t_m = -\frac{C}{A}$ for every $m$), and whether $B=0$ depends only on $H$.
\end{Observation}

\begin{Remark} \label{remark-curveasymptotic}
In particular, if $g_X \geq 2$ or $g_X=0$, then depending on $H$, either of two possibilities may occur.  If $B=0$ for every $Y$, then a given line bundle $L$ is $H^-$-semistable if and only if $L$ is balanced in the sense of \cite{caporaso} (Definition \ref{def-balanced}).  If $B \neq 0$ for some $Y$, then no line bundle is $H^-$-semistable.  (Similar statements hold for $H^+$-semistability.)  This says that for curves, the line bundle $H$ contributes nothing to the question of semistability.
\end{Remark}

The conclusions of Theorem \ref{thm-curve} follow from Observations \ref{obs-genus} and \ref{obs-B=0} together with the easy calculation that if $H = K_X$, then $B=0$ for any $Y$.

\begin{Remark} \label{remark-curvestrict}
This proof shows that for curves of genus not equal to $1$, the unit interval of Definition \ref{def-**} contains two integers if and only if $-\frac{C}{A}$ is an integer.
\end{Remark}

\vspace{.2in}

\section{Varieties of arbitrary dimension} \label{section-arbitrary}
\setcounter{Theorem}{0}

In this section we prove the following:

\begin{Theorem} \label{thm-dimd}
Let $X$ be a variety of dimension $d$ that has a smooth deformation, and let $Y$ be a union of irreducible components of $X$.  Let $K_X$ denote the canonical bundle of $X$.  If $[K_X^d] [(K_X|_Y)^{d-1}Y^2] < 0$, then $(X,Y)$ is $K_X^-$-twistable.  If $[K_X^d] [(K_X|_Y)^{d-1}Y^2] > 0$, then $(X,Y)$ is $K_X^+$-twistable.
\end{Theorem}

\begin{Remark}
For example, the quantity $[K_X^d] [(K_X|_Y)^{d-1}Y^2]$ is negative whenever $K_X$ is ample, and positive whenever $K_X^\vee$ is ample.  This is not the only situation of interest, of course; for instance, if $X$ is the special fiber in a blowup as described in Section \ref{subsection-background}, then the restriction of $K_X$ to the `exceptional components' of $X$ may be trivial, yet the hypotheses of Theorem \ref{thm-dimd} may still hold for the values of $Y$ required for the application of Theorem \ref{thm-tree}.
\end{Remark}

Applying Theorem \ref{thm-dimd} to Theorem \ref{thm-tree} yields the following answer to Question \ref{question-twists}:
\begin{Theorem} \label{thm-upshot}
Let $X$ be a variety of dimension $d$ that has a smooth deformation, and assume that its dual graph $\Gamma_X$ is a tree.  Let $K_X$ denote the canonical bundle of $X$.  Suppose that for every union of irreducible components $Y \subset X$ such that both $Y$ and $\overline{X \backslash Y}$ are connected, we have 
\[
[K_X^d] [(K_X|_Y)^{d-1}Y^2] < 0 \text{ (resp. $[K_X^d] [(K_X|_Y)^{d-1}Y^2] > 0$).}
\]  
Then any line bundle $L$ on $X$ has at least one $K_X^-$-semistable (resp. $K_X^+$-semistable) twist, and at most $2^{n-1}$, where $n$ is the number of irreducible components of $X$.
\end{Theorem}

\vspace{.2in}

\subsection{Proof of Theorem \ref{thm-dimd}}

The proof of Theorem \ref{thm-dimd} rests on the following observations, which will be proved in Section \ref{subsection-lemmas}.  First, let $H$ be any line bundle on $X$.  Let us write
\begin{align*}
e_Y(L+mH+aY) 
&= \tilde{A}_d(m) a^d + \tilde{A}_{d-1}(m) a^{d-1} + \ldots + \tilde{A}_1(m) a + \tilde{A}_0(m) \text{,}
\end{align*}
where for each $i$, $\tilde{A}_i(m)$ is a polynomial in $m$.  For each $m$, let $r_1(m), \ldots, r_{d_m}(m)$ denote the roots of $e_Y(L+mH+aY)$ (viewed as a polynomial in $a$, of degree $d_m \leq d$), counted with multiplicity; we fix this numbering of the roots for the rest of this section.

\begin{Lemma} \label{lemma-coefficients}
For each $i$, $\tilde{A}_i(m)$ has degree at most $2d-1-i$ in $m$.
\end{Lemma}

If we take $H=K_X$, then a direct calculation shows that more is true:

\begin{Lemma} \label{lemma-degree}
When $H = K_X$, we have $\deg \tilde{A}_0(m) \leq 2d-2$.  Furthermore, if $H=K_X$ and $[K_X^d] [(K_X|_Y)^{d-1}Y^2] \neq 0$, we have $\deg \tilde{A}_1(m) \geq 2d-2$.
\end{Lemma}

Motivated by these lemmas, let us set $p_m(a) = \frac{1}{m^{2d-2}} e_Y(L+mH+aY)$ (this of course has the same roots as $e_Y(L+mH+aY)$), and let $A_i(m) = \frac{1}{m^{2d-2}} \tilde{A}_i(m)$ for each $i$.  Let $A_i = \lim_{m \rightarrow \infty} A_i(m)$.  By Lemma \ref{lemma-coefficients}, $A_i=0$ for each $i \geq 2$.  Moreover, if $H=K_X$ and $[K_X^d] [(K_X|_Y)^{d-1}Y^2] \neq 0$, then by Lemma \ref{lemma-degree}, $A_0$ is finite and $A_1 \neq 0$.  In this situation, let $p(a) = A_1 a + A_0$, and let $s = -\frac{A_0}{A_1}$.

The linear polynomial $p(a)$ plays the same role as the linear polynomial $Ab+C$ of Section \ref{section-curve}.  Roughly speaking, for large $m$, $p_m(a)$ will behave like $p(a)$: there will be a unit interval comprising the simultaneous solutions to $p_m(a) \leq 0$ and $p_m(a+1) \geq 0$ (which are precisely the simultaneous solutions to $e_Y(L+mH+aY) \leq 0$ and $e_Z(L+mH+aY) \leq 0$, because of Proposition \ref{prop-identity}), and as $m$ grows the sequence of unit intervals will converge.  We can detect this behavior in the fact that there is essentially a unique sequence of roots of the polynomials $p_m(a)$ converging to $s$, while all the other roots grow unboundedly with $m$.

Lemmas \ref{lemma-uniformconvergence}-\ref{lemma-diverging} make these remarks precise.  For all of them, we use the notation above, taking $H=K_X$ and assuming $[K_X^d] [(K_X|_Y)^{d-1}Y^2] \neq 0$.

\begin{Lemma} \label{lemma-uniformconvergence}
The polynomials $p_m(a)$ converge uniformly to $p(a)$ on bounded subsets of $\mathbf{R}$.
\end{Lemma}

\begin{Lemma} \label{lemma-rootsequence}
There exists a sequence $(i_m)$ such that $\lim_{m \rightarrow \infty} r_{i_m}(m) = s$.
\end{Lemma}

\begin{Lemma} \label{lemma-monotone}
If $(i_m)$ is as in Lemma \ref{lemma-rootsequence}, then either $r_{i_m}(m) \geq s$ for all sufficiently large $m$ or $r_{i_m}(m) \leq s$ for all sufficiently large $m$.
\end{Lemma}

\begin{Lemma} \label{lemma-diverging}
For any $M>0$, there exists $N>0$ such that for any $m>N$, there is at most one $j \in \{1,\ldots,d_m\}$ such that $r_j(m) \in [-M,M]$.
\end{Lemma}

\begin{proof}[Proof of Theorem \ref{thm-dimd}.]
Suppose $[K_X^d] [(K_X|_Y)^{d-1}Y^2] < 0$.  Then by the calculations in the proof of Lemma \ref{lemma-degree}, we have $A_1>0$.  Let $(i_m)$ be as in Lemma \ref{lemma-rootsequence}, and for each $m$ let $s_m = r_{i_m}(m)$.  Let 
\[
I = \{a \in \mathbf{R}: p_m(a) \leq 0, p_m(a+1) \geq 0 \text{ for all sufficiently large } m\} \text{.}
\]
$I$ is precisely the set which we would like to prove is a unit interval containing at least one of its endpoints.  (This is because for any $a \in \mathbf{R}$, $p_m(a)$ has the same sign as $e_Y(L+mK_X+aY)$, and $p_m(a+1)$ has the same sign as $e_Y(L+mK_X+(a+1)Y)$, which equals $-e_Z(L+mK_X+aY)$ by Proposition \ref{prop-identity}.)

From Lemma \ref{lemma-uniformconvergence}, Lemma \ref{lemma-diverging}, and the fact that $A_1>0$, we have that for every sufficiently large $m$, $p_m(a)$ has exactly one root in $(s-1,s+1)$, namely $s_m$; $p_m(a) < 0$ for every $a \in (s-1,s_m)$, and $p_m(a) > 0$ for every $a \in (s_m,s+1)$.  Since $\lim s_m = s$, this implies that for each $a \in (s-1,s)$, $p_m(a)<0$ and $p_m(a+1)>0$ for every sufficiently large $m$.

Furthermore, for each $a \notin [s-1,s]$, Lemma \ref{lemma-uniformconvergence} implies that for every sufficiently large $m$, $p_m(a)$ and $p_m(a+1)$ are either both strictly positive or both strictly negative, since the same is true for $p(a)$ and $p(a+1)$.

Therefore $(s-1,s) \subset I \subset [s-1,s]$.  It remains to check that $I$ contains at least one of its endpoints.  By Lemma \ref{lemma-monotone}, one of the following cases must occur:
	\begin{itemize}
	\item Case 1: $s_m \geq s$ for all sufficiently large $m$, and there are infinitely many $m$ such that $s_m \neq s$.  Then $I=(s-1,s]$.
	\item Case 2: $s_m \leq s$ for all sufficiently large $m$, and there are infinitely many $m$ such that $s_m \neq s$.  Then $I=[s-1,s)$.
	\item Case 3: $s_m = s$ for all sufficiently large $m$.  Then $I=[s-1,s]$.
	\end{itemize}

When $[K_X^d] [(K_X|_Y)^{d-1}Y^2] > 0$, the proof is identical except that in that case $A_1<0$.
\end{proof}

\begin{Remark} \label{remark-dimd-strictss}
This proof shows that the unit interval of Definition \ref{def-**} contains two integers (so we have `strict semistability') if and only if Case 3 occurs and $s \in \mathbf{Z}$.
\end{Remark}

\begin{Remark} \label{remark-surfaceasymptotic}
In Remark \ref{remark-curveasymptotic} we observed that for curves the line bundle $H$ plays no real role.  This is in contrast to the higher-dimensional case.  Suppose, in the situation of Theorem \ref{thm-dimd}, that $e_Y(L) \leq 0$ iff $e_Y(L+K_X) \leq 0$; that is, the sign of $p_m(a)$ is constant with respect to $m$.  Then, since the $p_m$ converge to the linear polynomial $p$, each $p_m$ must also be linear, with the same root as $p$; that is, $A_2(m), \ldots, A_d(m)$ are identically zero.

This is not always the case, however; for example, when $d=2$ and $H=K_X$, 
\[
\tilde{A}_2(m) = -[Y^3][K_X^2] m - [Y^3][LK_X]+2[Y^3]\chi(X,O_X) \text{,}
\]
and this cannot be identically zero unless $[Y^3] = 0$ or $[K_X^2] = 0$.
\end{Remark}

\vspace{.2in}

\subsection{Proofs of Lemmas \ref{lemma-coefficients}-\ref{lemma-diverging}} \label{subsection-lemmas}

\begin{proof}[Proof of Lemma \ref{lemma-coefficients}.]
Let 
\begin{align*}
q_X(L+mH) &= d! \chi(X,L+mH)-[(L+mH)^dX] \text{,}
\\
q_Y(L+mH+aY) &= d! \chi(Y,L+mH+aY)-[(L+mH+aY)^dY] \text{;}
\end{align*}
these are both polynomials of degree $d-1$ in their arguments.  By definition,
\begin{align*}
e_Y(L+mH+aY) 
&= [(L+mH)^dX] \cdot \frac{1}{d+1} \sum_{j=2}^{d+1} \binom{d+1}{j} (-1)^{j-1} [(L+mH+aY)^{d+1-j} Y^{j}] 
\\
&+ q_X(L+mH) \cdot \frac{1}{d+1} \sum_{j=1}^{d+1} \binom{d+1}{j} (-1)^{j-1} [(L+mH+aY)^{d+1-j} Y^{j}] 
\\
&- [(L+mH)^d X] \cdot q_Y(L+mH+aY) \text{.}
\end{align*}
In the first and third summands (resp. the second summand), the first factor has degree at most $d$ (resp. $d-1$) in $m$, while in the second factor the coefficient of $a^i$ has degree at most $d-1-i$ (resp. $d-i$) in $m$.
\end{proof} 

\begin{proof}[Proof of Lemma \ref{lemma-degree}.]
By the Hirzebruch-Riemann-Roch formula, for any line bundle $H$, we have
\begin{align*}
d! \chi(X,L+mH) &= [(L+mH)^dX] -\frac{d}{2}[(L+mH)^{d-1}K_X] + \epsilon_X(m) \text{,}
\\
d! \chi(Y,L+mH+aY) &= [(L+mH+aY)^dY] -\frac{d}{2}[(L+mH+aY)^{d-1}K_Y] + \epsilon_Y(m) \text{,}
\end{align*}
where $\epsilon_X(m)$ and $\epsilon_Y(m)$ are polynomials of degree at most $d-2$ in $m$.  
If we substitute these expressions into $e_Y(L+mH+aY)$, we see that in $\tilde{A}_1(m)$, the coefficient of $m^{2d-2}$ is 
\[
[H^dX]\left(-\frac{d}{2} (d-1) [H^{d-2}Y^3] + \frac{d}{2} (d-1) [H^{d-2}YK_Y] \right) 
-\frac{d}{2} [H^{d-1}K_X] \cdot d [H^{d-1}Y^2] \text{,}
\]
while in $\tilde{A}_0(m)$, the coefficient of $m^{2d-1}$ is
\[
[H^dX] \left(-\frac{d}{2} [H^{d-1}Y^2] + \frac{d}{2} [H^{d-1}K_Y] \right) 
-\frac{d}{2} [H^{d-1}K_X] [H^d Y] \text{.}
\]
If we take $H=K_X$ and simplify using the formula $K_X|_Y = K_Y-Y^2$, we find that the coefficient of $m^{2d-2}$ in $\tilde{A}_1(m)$ is
\[
-\frac{d}{2} [K_X^d] [(K_X|_Y)^{d-1}Y^2] \text{,}
\]
while the coefficient of $m^{2d-1}$ in $\tilde{A}_0(m)$ is $0$.
\end{proof}

\begin{proof}[Proof of Lemma \ref{lemma-uniformconvergence}.]
Fix any $M \geq 1$ and any $\epsilon > 0$.  Choose $N$ such that \[ |A_i(m)-A_i|<\frac{\epsilon}{(d+1)} M^{-d} \] for every $m > N$ and every $i$; then for any $a \in [-M,M]$, $|p_m(a)-p(a)| < \epsilon$.  Therefore $p_m(a)$ converges uniformly to $p(a)$ on $[-M,M]$.
\end{proof}

\begin{proof}[Proof of Lemma \ref{lemma-rootsequence}.]
Fix any $M>2|s|$ and any $\epsilon>0$.  Let $N$ be such that \[ |p_m(a)-p(a)|<\epsilon \cdot |A_1| \] for any $m>N$ and any $a \in [-M,M]$.  Then whenever $a \in [-M,M]$ and $m>N$, $p_m(a)$ has one sign (negative if $A_1>0$, positive if $A_1<0$) if $a-s \geq \epsilon$, and the opposite sign (positive if $A_1>0$, negative if $A_1<0$) if $s-a \geq \epsilon$.  Therefore $p_m$ has a root in $(s-\epsilon, s+\epsilon)$.
\end{proof}

\begin{proof}[Proof of Lemma \ref{lemma-monotone}.]
First note that either $p_m(s)>0$ for all sufficiently large $m$, $p_m(s)<0$ for all sufficiently large $m$, or $p_m(s)=0$ for all $m$.  This is because $p_m(s) = m^{2-2d} \cdot q(m)$ where $q(m)$ is a polynomial in $m$, and we must have either $\lim_{m \rightarrow \infty} q(m) = \pm \infty$ or $q(m)=0$ for all $m$.

Now observe that the derivative $p_m'(a)$ converges uniformly to $p'(a) = A_1$ on bounded subsets of $\mathbf{R}$.  In particular, for all sufficiently large $m$, we have \[ |p_m'(a)-A_1|<\frac12 |A_1| \] for every $a \in (s-1,s+1)$.

Therefore:
\begin{itemize}
\item If $\lim_{m \rightarrow \infty} q(m)=\infty$ and $A_1>0$, then for every sufficiently large $m$, $p_m(s)>0$ and $p_m'(a)>0$ for all $a \in (s-1,s+1)$, so $r_{i_m}(m)<s$.
\item If $\lim_{m \rightarrow \infty} q(m)=\infty$ and $A_1<0$, then for every sufficiently large $m$, $p_m(s)>0$ and $p_m'(a)<0$ for all $a \in (s-1,s+1)$, so $r_{i_m}(m)>s$.
\item If $\lim_{m \rightarrow \infty} q(m)=-\infty$ and $A_1>0$, then for every sufficiently large $m$, $p_m(s)<0$ and $p_m'(a)>0$ for all $a \in (s-1,s+1)$, so $r_{i_m}(m)>s$.
\item If $\lim_{m \rightarrow \infty} q(m)=-\infty$ and $A_1<0$, then for every sufficiently large $m$, $p_m(s)<0$ and $p_m'(a)<0$ for all $a \in (s-1,s+1)$, so $r_{i_m}(m)<s$.
\item If $q(m)=0$ for every $m$, then for every sufficiently large $m$, $r_{i_m}(m)=s$.
\end{itemize}

\end{proof}

\begin{proof}[Proof of Lemma \ref{lemma-diverging}.]
Suppose for the sake of contradiction that for each $k \in \mathbf{N}$ we can find $m_k>k$ such that there are distinct $j, j'$ such that $t_k := r_j(m_k)$ and $u_k := r_{j'}(m_k)$ are both in $[-M,M]$.  We may assume that $(m_k)$ is an increasing sequence.  For each $k$, we have 
	\[
	p_{m_k}(a) = (a-t_k)(a-u_k) q_k(a) \text{,}
	\]
	where
	\[
	q_k(a) = B_{d-2}(k)a^{d-2} + \ldots + B_1(k)a + B_0(k)
	\]
	for some polynomials $B_0(k), \ldots, B_{d-2}(k)$.  Expanding, and equating coefficients of $a$, we find that
	\begin{align*}
	A_d(m_k) &= B_{d-2}(k) \text{,}
	\\
	A_{d-1}(m_k) &= B_{d-3}(k) -(t_k+u_k)B_{d-2}(k) \text{,}
	\\
	A_i(m_k) &= B_{i-2}(k) - (t_k+u_k)B_{i-1}(k) + t_k u_k B_{i}(k) 
	\hspace{.1in} \text{ for each $i \in \{2,\ldots,d-2\}$,}
	\\
	A_1(m_k) &= -(t_k+u_k)B_0(k) + t_k u_k B_1(k) \text{,}
	\\
	A_0(m_k) &= t_k u_k B_0(k) \text{.}
	\end{align*}
	Since the sequences $(t_k)$ and $(u_k)$ are bounded, and since $\lim_{m \rightarrow \infty} A_i(m) = 0$ for every $i \geq 2$, starting from the top line and working our way down we deduce that $\lim_{k \rightarrow \infty} B_i(k) = 0$ for every $i$.  But this contradicts the fact that $\lim_{m \rightarrow \infty} A_1(m) \neq 0$.

\end{proof}


\begin{thebibliography}{99}

\bibitem{alper} J. Alper. 
\emph{Good moduli spaces for Artin stacks}. 
Ann. Inst. Fourier 63 (2013), no. 6, 2349-2402.

\bibitem{altmankleiman} A. Altman and S. Kleiman.
\emph{Compactifying the Picard scheme}. 
Adv. in Math. 35 (1980), no. 1, 50-112.

\bibitem{caporaso} L. Caporaso. 
\emph{A compactification of the universal Picard variety over the moduli space of stable curves}.
J. Amer. Math. Soc. 7 (1994), no. 3, 589-660.

\bibitem{k3} A. Chowdhury. 
\emph{Stability of line bundles on K3 surfaces of Type II}.  In preparation.

\bibitem{esteves} E. Esteves. 
\emph{Compactifying the relative Jacobian over families of reduced curves.}
Trans. Amer. Math. Soc. 353 (2001), 3045-3095.

\bibitem{estevespacini} E. Esteves, M. Pacini.  
\emph{Semistable modifications of families of curves and compactified Jacobians}.  
Available at arXiv:1406.1239, 2014.

\bibitem{gieseker82} D. Gieseker.  \emph{Lectures on moduli of curves}.  Springer, 1982.

\bibitem{gieseker77} D. Gieseker.  \emph{On the moduli of vector bundles on an algebraic surface}.  
Ann. of Math. (2) 106 (1977), no. 1, 45-60.

\bibitem{mtt} D. Markushevich, A. Tikhomirov, G. Trautmann.  \emph{Bubble tree compactification of moduli spaces of vector bundles on surfaces}.  Centr. Eur. J. Math. 10 (2012), no. 4, 1331-1355.

\bibitem{maruyama} M. Maruyama.  \emph{Moduli of stable sheaves II}. J. Math. Kyoto Univ. 18 (1978), no. 3, 557-614.

\bibitem{nagarajseshadri} D. Nagaraj, C. Seshadri.  \emph{Degenerations of the moduli spaces of vector bundles on curves I}. Proc. Indian Acad. Sci. Math. Sci. 107 (1997), 101-137.

\bibitem{odaseshadri} T. Oda, C. Seshadri.
\emph{Compactifications of the generalized Jacobian variety}.
Trans. Amer. Math. Soc. 253 (1979), 1-90.

\bibitem{osserman} B. Osserman. \emph{Limit linear series for curves not of compact type}.  Available at arXiv:1406.6699, 2014.

\bibitem{pandharipande} R. Pandharipande.
\emph{A compactification over $\overline{M}_g$ of the universal moduli space of slope-semistable vector bundles}.
J. Amer. Math. Soc. 9 (1996), no. 2, 425-471.

\bibitem{schmitt} A. Schmitt.
\emph{The Hilbert compactification of the universal moduli space of semistable vector bundles over smooth curves}.
J. Differential Geom. 66 (2004), no. 2, 169-209.

\bibitem{timofeeva} N. Timofeeva.  \emph{On a new compactification of the moduli of vector bundles on a surface}.  Sbornik: Mathematics 199 (2008), no. 7, 1051-1070.

\end{thebibliography}
\end{document}